\documentclass[12pt,a4paper]{article} 
\usepackage{amsmath,amssymb,amsthm,amsfonts,amscd,euscript,verbatim, t1enc, newlfont}
\usepackage{hyperref}
\usepackage{yfonts}
\usepackage{graphicx}
\usepackage[all]{xy}
\theoremstyle{definition}
\newtheorem{theo}{Theorem}[subsection]

\newtheorem{theore}{Theorem}[section]
\newtheorem{pr}[theo]{Proposition}

 \newtheorem{lem}[theo]{Lemma}
 \newtheorem{coro}[theo]{Corollary}
  
    \newtheorem{question}[theore]{Question}
\theoremstyle{remark}
\newtheorem{rema}[theo]{Remark}

\theoremstyle{definition}
\newtheorem{defi}[theo]{Definition}

\newcommand\cu{\underline{C}}
\newcommand\du{\underline{D}}

\newcommand\au{\underline{A}}

\newcommand\q{{\mathbb{Q}}}
\newcommand\obj{\operatorname{Obj}}
\newcommand\hw{{\underline{Hw}}}

\newcommand\zop{{\mathbb{Z}[\frac{1}{p}]}}
\newcommand\n{{\mathbb{N}}}
\newcommand\z{{\mathbb{Z}}}

\newcommand\ns{\{0\}}

\newcommand\id{\operatorname{id}}
 \newcommand\lan{\langle}
\newcommand\ra{\rangle}

\DeclareMathOperator\kar{\operatorname{Kar}}

\DeclareMathOperator\co{\operatorname{Cone}}
\DeclareMathOperator\mo{\operatorname{Mor}}

\DeclareMathOperator\inli{\varinjlim}

\newcommand\hu{\underline{H}}

\newcommand\dmgep
{{DM^{\,\, {eff}}_{\scalebox{0.6}{gm},\zop}}{}}

\newcommand\dmgepr 
{{DM^{\,\, {eff}}_{\scalebox{0.5}{gm},\zop,R}}{}}

\newcommand\dmgepq 
{{DM^{\,\, {eff}}_{\scalebox{0.5}{gm},\zop,\q}}{}}

\numberwithin{equation}{subsection}

\begin{document}

\title
 {On the weight lifting property for localizations of triangulated categories}
 \author{
Mikhail V. Bondarko, Vladimir A. Sosnilo\thanks{Research is supported by the Russian Science Foundation grant no. 16-11-10073.
}}
 \maketitle
\begin{abstract}
 As we proved earlier, for a triangulated category $\cu$ endowed with a weight structure $w$ and a triangulated subcategory $\du$ of $\cu$ (strongly) generated by 
cones of a set of morphisms $S$ in the heart $\hw$ of $w$  there exists a weight structure $w'$   on the Verdier quotient  $\cu'=\cu/\du$ such that 
 the localization functor  $\cu 
\to \cu'$ is weight-exact (i.e., "respects weights"). 
The goal of this paper is to find  conditions ensuring that  for any object of $\cu'$ of non-negative (resp. non-positive) weights there exists its preimage in $\cu$ satisfying the same condition; we call  a certain stronger version of the 
latter assumption 
 the left (resp., right) 
{\it weight lifting property}. We prove that these  weight lifting properties are fulfilled whenever the set $S$ 
satisfies the 
 corresponding (left or right) 
 Ore conditions. Moreover, if $\du$ is generated by objects of  $\hw$ then any object of $\hw'$ lifts to $\hw$.
	
We apply these results to obtain some new results on Tate motives and finite spectra (in the stable homotopy category). 
 Our results are also applied to the study of the so-called Chow-weight homology in another paper.

\end{abstract}

\tableofcontents

\section{Introduction}


Let $X$ be a finite spectrum. Assume that the stable homotopy groups $\pi^{s}_i(X)$ 
are torsion for $i\le n$ for some $n \in \z$. 
It is natural to ask whether $X$ belongs to the extension-closure of the union of the class of
finite torsion spectra with the class of finite $n$-connected spectra. 
Below we will establish this statement in the following strong form: there exists a distinguished triangle $T \to X \to M \to T[1]$ in $SH^{fin}$ such that 
 $T$ is torsion and $M$ is an $n$-connected  spectrum.


In this paper we 
 prove a generalization of the aforementioned result that turns out to be important for studying the 
properties of the so-called Chow-weight homology theories.

To state the question we will use the language of weight structures; this notion was  
introduced by 
 the first author and independently by D. Pauksztello (under the name of co-$t$-structures). Weight structures on triangulated categories are important counterparts of $t$-structures; similarly to $t$-structures, weight structures on a triangulated category $\cu$ are defined in terms of 
classes of objects $\cu_{w\ge 0}$ and $\cu_{w\le 0}$ satisfying certain axioms (see Definition \ref{dwstr}).  We will say that objects 	of the first (resp., of the second) class have non-negative  (resp., of non-positive) weights.

Let $\cu$ be a triangulated category endowed with a weight structure $w$, $\du$ be its subcategory such that the localized category $\cu'=\cu/\du$ possesses a weight structure $w'$ 
for which the localization functor $L$ is weight-exact (i.e., preserves objects of non-positive and of non-negative weights). 
Then $\cu'_{w'\ge 0}$ equals the class $ \kar_{\cu'}(L(\cu_{w\ge 0}))$ and $\cu'_{w'\le 0} = \kar_{\cu'}(L(\cu_{w\le 0}))$ (see 
 Remark 2.1.5(6) of \cite{bpure} or Proposition 3.1.1 of \cite{bsnew}; here $\kar_{\cu'}(A)$ denotes the class  of $\cu'$-retracts of  elements of a class  $A\subset \obj \cu'$). 
Yet to prove the statement formulated above it is important to know whether  we can avoid the Karoubi-closures  here, i.e., whether $\cu'_{w'\ge 0}$ essentially equals the image $L(\cu_{w\ge 0}$ (resp., whether $\cu'_{w'\le 0}=L(\cu_{w\le 0}$). 
 So we want to ask whether  any object of $\cu'$ whose "weights" satisfy a certain inequality lifts to an object of $\cu$ satisfying the same weight bound.\footnote{Note also: since  $\cu'_{w'\ge 0}= \kar_{\cu'}(L(\cu_{w\ge 0}))$,  $\cu'_{w'\ge 0}$  (essentially) equals  $L(\cu_{w\ge 0})$ if and only if  for any $X\in \cu_{w\ge 0}$ all $\cu'$- retracts of $L(X)$  lift to elements of $\cu_{w\ge 0}$, i.e.,  $L(\cu_{w\ge 0})$ should be (essentially) {\it Karoubi-closed} in $\cu'$.}

It turns out to be reasonable to ask this question in the following stronger form.

\begin{question}\label{q1}
For $X\in \obj \cu$ assume that $L(X)$ belongs to $\cu'_{w'\ge 0}$. 
Does there exist a distinguished triangle $T\to X\to X_{\ge 0}\to T[1]$ such that $T\in \obj \du$ and $X_{\ge 0}\in \cu_{w\ge 0}$.
\end{question}

If the answer to this question is positive for all $X$ (such that $L(X)\in \cu'_{w'\ge 0}$)  then we will say that the triple $(\cu,w,\du)$ satisfies the {\it left weight lifting property}.\footnote{Note that the localization functor $L:\cu\to \cu'$ "will not change" if we replace $\du$ by its dense triangulated subcategory; still this replacement may affect the answer to Question \ref{q1}.}

Dually, $(\cu,w,\du)$ is said to satisfy the {\it right weight lifting property} if for any $X$ such that $L(X) \in \cu'_{w'\le 0}$ there exists a distinguished triangle 
$T\to X_{\le 0}\to X\to T$ such that $T\in \obj \du$ and $X_{\le 0}\in \cu_{w\le 0}$.

Now let $\du$ be the full triangulated subcategory of $\cu$ strongly generated by cones of 
 some set $S$ of $\hw$-morphisms. 
Then the category $\cu'=\cu/\du$ 
 does possess a weight structure such that 
$L$ is weight-exact (see Theorem 4.2.3(1) of \cite{bososn}). 
Below we will prove that $L$ satisfies the  left weight lifting property whenever $S$ satisfies the left Ore condition (see Theorem \ref{dirsum} below); the corresponding "right" version of this statement is also valid (being the categorical dual of the "left" one). 
This statement certainly can be applied to the 
 stable homotopy category and to 
 Tate motives (see Remark \ref{tappl}). We obtain the statements mentioned in the beginning of this section as particular cases.
Another useful  corollary is as follows: any object of $\hw'$ lifts to an object of $\hw$ if $\du$ is a bounded subcategory of $\cu$ such that $w$ restricts to it; 
it was applied to certain "relative motives" in 
\cite{bondiv}. 
Furthermore, our theorem becomes very useful for the study of the Chow-weight homology groups (see Theorems 3.2.3 and  4.2.5  of \cite{bosocwh}).

Next, for an object $X$ of $\cu$ and a 
 sequence of subcategories $\dots \subset \du_{i+1} \subset\du_i\subset \dots \subset \cu$ such that  $w$ restricts to them (note that in this case there exist weight structures on the localizations $\cu/\du_i$ such that all the localization functors $\cu=\to \cu/\du_i$ are weight-exact), one can study the "refined" weights of $X$, i.e., weights of $X$ modulo  each of  $\du_i$. 
 Let $S$ be a set of morphisms in the heart of $w$ and $K$ be a triangulated subcategory generated by cones of $S$. One can consider the subcategories $\du'_i$  of $\cu/K$ generated by the images of $\du_i$ to get 
 certain refined weights for $\cu/K$. 
We would like to know whether an object $X$ of $\cu$ having certain weights in all of the localizations $(\cu/K)/\du'_i$ belongs to the extension closure of the union of $K$ and the set of objects having the corresponding weights in $\cu/\du_i$. 
 In \ref{apppl} we give  a 
certain affirmative answer to the latter question in the case  where $S$ is the set of multiplications by some integers.

We apply 
 the corresponding  Proposition \ref{torsindep} 
  to the study of Tate motives and their weight $0$ motivic homology (see Remark \ref{tateappl}(3)). 
In Theorem 4.2.5   of \cite{bosocwh} we also  apply this proposition to the case $\cu = \dmgep$ and $\du_i = \dmgep(i)$; this yields quite 
non-trivial properties of the so-called Chow-weight homology groups. 

Now we describe the structure of the paper; some more information of this sort can be found in the beginnings of sections.

In Section \S\ref{prelim} we recall some basics on weight structures. In particular, we recall certain results about localizations of triangulated categories with weight structures (see Proposition 
 \ref{p2loc}). 

Section \S\ref{oretwlp} is dedicated to the proof of Theorem \ref{dirsum} (which is the central result of this paper). It states that the localization functor $\cu \to \cu/\du$ satisfies the  left weight lifting property whenever $\du$ is  strongly generated by  cones of some set  $S \subset \mo \hw$ satisfying the left Ore condition. 
We also prove that objects of the heart of $\cu/\du$ lift to $\hw$ whenever $\du$ is strongly generated by a subset of $\cu_{w=0}$.

In \S\ref{sappl} we  give some applications of our theorem. 
First we consider the stable homotopy category and the category of Tate  motives. 
Next we compare the "refined" weights of  $X\in \obj \cu$ in a system of localizations $\cu/\du_i$ with  its weights in the categories $(\cu/\du_i) \otimes \z[J^{-1}]$ (where $\du_i$ are some triangulated subcategories to which $w$ restricts). 

The second author thanks Sergey O. Ivanov for his interesting comments concerning the weight lifting property and Sergey Sinchuk for correcting the 
 English.

\section{Preliminaries}\label{prelim}

In \S\ref{snotata} we introduce some notation and conventions related (mostly) to triangulated categories.

In \S\ref{ssws} we recall some basics on weight structures.

In \S\ref{swesloc}  we recall 
 in addition how weight structures "behave under localizations".

\subsection{Some (categorical) notation and lemmas}\label{snotata}

\begin{itemize}


\item Given a category $C$ and  $X,Y\in\obj C$, we denote by
$C(X,Y)$ the set of morphisms from $X$ to $Y$ in $C$.

\item For categories $C',C$ we write $C'\subset C$ if $C'$ is a full 
subcategory of $C$.

\item Given a category $C$ and  $X,Y\in\obj C$, we say that $X$ is a {\it
retract} of $Y$ 
 if $\id_X$ can be 
 factored through $X$.\footnote{If $C$ is triangulated or abelian,
then $X$ is a retract of $Y$ if and only if $X$ is its direct summand.}

\item An additive subcategory $\hu$ of additive category $C$ 
is called {\it Karoubi-closed}
  in $C$ if it
contains all retracts of its objects in $C$.
The full subcategory $\kar_{C}(\hu)$ of additive category $C$ whose objects
are all retracts of objects of a subcategory $\hu$ (in $C$) will be
called the {\it Karoubi-closure} of $\hu$ in $C$. 

\item The {\it Karoubization} $\kar(B)$ (no lower index) of an additive
category $B$ is the category of ``formal images'' of idempotents in $B$.
So, its objects are pairs $(A,p)$ for $A\in \obj B,\ p\in B(A,A),\ p^2=p$, and the morphisms are given by the formula 
\begin{equation}\label{mthen}
\kar(B)((X,p),(X',p'))=\{f\in B(X,X'):\ p'\circ f=f \circ p=f \}.\end{equation}
 The correspondence  $A\mapsto (A,\id_A)$ (for $A\in \obj B$) fully embeds $B$ into $\kar(B)$.
 Besides, $\kar(B)$ is {\it Karoubian}, i.e., 
 any idempotent morphism yields a direct sum decomposition in 
 $\kar(B)$. Equivalently, $B$ is Karoubian if (and only if) the canonical embedding $B \to \kar
(B)$ is an equivalence of categories.
 Recall also that $\kar(B)$ is canonically triangulated if $B$ is (see \cite{bashli}).

\item $\cu$ below will always denote some triangulated category;
usually it will
be endowed with a weight structure $w$. 

\item For any  $A,B,C \in \obj\cu$ we will say that $C$ is an {\it extension} of $B$ by $A$ if there exists a distinguished triangle $A \to C \to B \to A[1]$.

\item For any $E,E' \subset \obj \cu$ we denote by $E \star E'$ the class of all  extensions of elements of $E'$ by elements of $E$.

\item A class $D\subset \obj \cu$ is said to be  {\it extension-closed}
    if it 
		is closed with respect to extensions and contains $0$. We will call the smallest extension-closed subclass 
of objects of $\cu$ that  contains a given class $B\subset \obj\cu$ 
  the {\it extension-closure} of $B$. 

\item Given a class $D$ of objects of $\cu$ we denote by $\lan D\ra$ the smallest full 
triangulated subcategory of $\cu$ containing $D$. We will also call  $\lan D\ra$  the triangulated category {\it strongly generated} by $D$. 

\item We will often localize $\cu$ by a triangulated subcategory $\du$, whereas $\du$ will often be equal to $\lan D\ra$. 
To avoid set-theoretic difficulties we will always assume in these occurrences that $\du$ is essentially small and the class of $\cu$-isomorphism classes of elements of $D$
is a set. Respectively, if $D$ consists of cones of a class of morphism $S$ then we will assume (when considering $\cu/\du$)  the class of $\cu$-isomorphism classes of elements of $S$ is a set also.


\item For $X,Y\in \obj \cu$ we will write $X\perp Y$ if $\cu(X,Y)=\ns$. For
$D,E\subset \obj \cu$ we write $D\perp E$ if $X\perp Y$ for all $X\in D,\
Y\in E$.
Given $D\subset\obj \cu$ we  denote by $D^\perp$ the class
$$\{Y\in \obj \cu:\ X\perp Y\ \forall X\in D\}.$$
Dually, ${}^\perp{}D$ is the class
$\{Y\in \obj \cu:\ Y\perp X\ \forall X\in D\}$.

\item Given $f\in\cu (X,Y)$, where $X,Y\in\obj\cu$, we will call the third vertex
of (any) distinguished triangle $X\stackrel{f}{\to}Y\to Z$ a {\it cone} of
$f$ (recall that different choices of cones are connected by
non-unique isomorphisms).

\item For an additive category $B$ we  denote by $K(B)$ the homotopy category
of (cohomological) complexes over $B$. Its full subcategory of
bounded (resp., bounded below) complexes will be denoted by $K^b(B)$ (resp., $K^-(B)$). We will write $X=(X^i)$ if $X^i$ are the terms of the complex $X$. 


\end{itemize}

\subsection{Weight structures: basics}\label{ssws}

\begin{defi}\label{dwstr}
 A pair of subclasses $\cu_{w\le 0},\cu_{w\ge 0}\subset\obj \cu$ 
will be said to define a weight
structure $w$ for a triangulated category  $\cu$ if 
they  satisfy the following conditions.

(i) $\cu_{w\ge 0},\cu_{w\le 0}$ are 
Karoubi-closed in $\cu$
(i.e., contain all $\cu$-retracts of their elements).

(ii) {\bf Semi-invariance with respect to translations.}

$\cu_{w\le 0}\subset \cu_{w\le 0}[1]$, $\cu_{w\ge 0}[1]\subset
\cu_{w\ge 0}$.

(iii) {\bf Orthogonality.}

$\cu_{w\le 0}\perp \cu_{w\ge 0}[1]$.

(iv) {\bf Weight decompositions}.

 For any $M\in\obj \cu$ there
exists a distinguished triangle
\begin{equation}\label{wd}
X\to M\to Y
{\to} X[1]
\end{equation} 
such that $X\in \cu_{w\le 0},\  Y\in \cu_{w\ge 0}[1]$.
\end{defi}

We will also need the following definitions.

\begin{defi}\label{dwso}

Let $i,j\in \z$; assume that $\cu$ is a triangulated category endowed with a weight structure $w$.

\begin{enumerate}
\item\label{iII} The category $\hw\subset \cu$ whose objects are
$\cu_{w=0}=\cu_{w\ge 0}\cap \cu_{w\le 0}$ and morphisms are $\hw(Z,T)=\cu(Z,T)$ for
$Z,T\in \cu_{w=0}$,
 is called the {\it heart} of 
$w$.

\item\label{iIII} $\cu_{w\ge i}$ (resp. $\cu_{w\le i}$, resp. $\cu_{w= i}$) will denote $\cu_{w\ge
0}[i]$ (resp. $\cu_{w\le 0}[i]$, resp. $\cu_{w= 0}[i]$).

\item\label{iIV} $\cu_{[i,j]}$  denotes $\cu_{w\ge i}\cap \cu_{w\le j}$; so, this class  equals $\ns$ if $i>j$.

Elements of $\cu_{[i,j]}$ are said to have length at most $j-i$ (for any $i,j\in \z$).

$\cu^b\subset \cu$ will be the category whose object class is $\cup_{i,j\in \z}\cu_{[i,j]}$.

\item\label{iV} We will  say that $(\cu,w)$ is {\it  bounded}  if $\cu^b=\cu$ (i.e., if $\cup_{i\in \z} \cu_{w\le i}=\obj \cu=\cup_{i\in \z} \cu_{w\ge i}$).

\item\label{iVI}  Let $\cu'$  be a triangulated category endowed with a weight structure  $w'$; let $F:\cu\to \cu'$ be an exact functor.

We will say that $F$ is {\it left weight-exact} (with respect to $w,w'$) if it maps $\cu_{w\le 0}$ to $\cu'_{w'\le 0}$; it will be called {\it right weight-exact} if it maps $\cu_{w\ge 0}$ to $\cu'_{w'\ge 0}$. $F$ is called {\it weight-exact} if it is both left and right weight-exact.

Let $\cu$ and $\cu'$ 
be triangulated categories endowed with
weight structures $w$ and
 $w'$, respectively; let $F:\cu\to \cu'$ be an exact functor.

\item\label{iVII}  Let $\du$ be a full triangulated subcategory of $\cu$.

We will say that $w$ {\it restricts} to $\du$ if the classes
$\obj \du\cap \cu_{w\le
 0}$ and $\obj \du\cap \cu_{w\ge
 0}$ give a weight structure  $w_{\du}$ on $\du$.

\end{enumerate}
\end{defi}

\begin{rema}\label{rstws}

1. A  simple (and yet quite useful) example of a weight structure comes from the stupid
filtration on 
the triangulated categories $K^b(B)\subset K(B)$ 
for an arbitrary additive category
 $B$; see  Remark 1.2.3(1) of \cite{bonspkar} for more detail. 


2. 
A weight decomposition (of any $M\in \obj\cu$) is (almost) never canonical. 

Yet for  $m\in \z$ we will often need some choice of a weight decomposition of $M[-m]$ shifted by $[m]$. So we consider a distinguished triangle \begin{equation}\label{ewd} w_{\le m}M\to M\to w_{\ge m+1}M \end{equation} 
for some $ w_{\ge m+1}M\in \cu_{w\ge m+1}$, $ w_{\le m}M\in \cu_{w\le m}$.  

 We will often use this notation below (though $w_{\ge m+1}M$ and $ w_{\le m}M$ are not canonically determined by $M$).
 
3. In the current paper we use the ``homological convention'' for weight structures; 
it was previously used in \cite{wildint}, \cite{bososn}, \cite{bonspkar},  and in \cite{bpure}, whereas in 
\cite{bws}  the ``cohomological" convention was used. In the latter convention 
the roles of $\cu_{w\le 0}$ and $\cu_{w\ge 0}$ are interchanged, i.e., one
considers   $\cu^{w\le 0}=\cu_{w\ge 0}$ and $\cu^{w\ge 0}=\cu_{w\le 0}$. 

Besides, in \cite{bws} both "halves" of $w$ were required to be  additive. Yet this additional restriction is easily seen to follow from the remaining axioms; see Remark 1.2.3(4) of \cite{bonspkar}.
 
\end{rema}

Let us recall some basics on weight structures. 

\begin{pr} \label{pbw}
Let $\cu$ be a triangulated category, $n\ge 0$; we will assume 
that $w$ is a fixed 
weight structure on it. 

\begin{enumerate}


\item \label{idual}
The axiomatics of weight structures is self-dual, i.e., for $\du=\cu^{op}$
(so $\obj\du=\obj\cu$) there exists the (opposite)  weight
structure $w'$ for which $\du_{w'\le 0}=\cu_{w\ge 0}$ and
$\du_{w'\ge 0}=\cu_{w\le 0}$.

 \item\label{iwd0} 
 If $M\in \cu_{w\ge -n}$ 
 then $w_{\le 0}M\in \cu_{[-n,0]}$.

\item\label{igenw0}
Assume $w$ is bounded below.
Then $\cu_{w\le 0}$ is the
   extension-closure of $\cup_{i\le 0}\cu_{w=i}$ in $\cu$;
   
Dually, if $w$ is bounded above then $\cu_{w\ge 0}$ is the
  extension-closure of $\cup_{i\ge 0} \cu_{w=i}$ in $\cu$.
	
	\item\label{i01}  $\cu_{[0,1]}$ consists exactly of cones of morphisms in $\hw$ (in $\cu$).

 \item\label{iort}
 $\cu_{w\ge 0}=(\cu_{w\le -1})^{\perp}$ and $\cu_{w\le 0}={}^{\perp} \cu_{w\ge 1}$; hence these classes are extension-closed.

\item\label{iwe} Assume that $w'$ is a weight structure
for a triangulated category $\cu'$ and that $w$ is bounded. Then an exact functor $F:\cu\to \cu'$ is
weight-exact if and only if $F(\cu_{w=0})\subset \cu'_{w'=0}$.

\end{enumerate}
\end{pr}
\begin{proof}


All of these statements 
 can be found in \cite{bws} (pay attention to Remark
\ref{rstws}(3)!). 

\end{proof}

\subsection{Weight structures in localizations (reminder)}\label{swesloc}

First we recall the following general fact.

\begin{pr}\label{pdesc}
Let $\cu$ and $\cu'$ be triangulated categories endowed with weight structures  $\cu$ and $\cu'$, respectively; let $F:\cu\to \cu'$ be an exact functor that is essentially surjective on objects.
 Then $F$ is  weight-exact with respect to $(w,w')$ if and only if $w'=(\kar_{\cu'}(F(\cu_{w\le 0})), \kar_{\cu'}(F(\cu_{w\ge 0}))$.
\end{pr}
\begin{proof}
This is an easy consequence of \cite[Proposition 1.2.4(10)]{bpure} (cf. Remark 2.1.5(6) of ibid.).

\end{proof}

\begin{rema}\label{rdesc}
This certainly means that there is at most one weight structure $w'$ such that $F$ is weight-exact; this weight structure exists whenever  $(\kar_{\cu'}(F(\cu_{w\le 0})), \kar_{\cu'}(F(\cu_{w\ge 0}))$ is a weight structure on $\cu'$.

In the latter case we will say that $w$ {\it descends} to $\cu'$, and the corresponding $w'$ is the {\it descended} weight structure. We will use the notation $w'$ for it whenever this will cause no ambiguity. 

Below we will be interested in the case where $F$ is some Verdier localization functor (certainly, $F$ is surjective on objects in this case).
\end{rema}

Now we recall that weight structures descend to localizations "surprisingly often".

\begin{pr}\label{p2loc}
Let $\du\subset \cu$ be a 
triangulated subcategory of
$\cu$ strongly generated by some set $B$ of elements of $\cu_{[0,1]}$. Denote the localization functor $\cu\to \cu'$ by $L$, where $\cu'= \cu/\du$ is the Verdier quotient of $\cu$ by $\du$.

For any $i,j\in\z\cup\{-\infty,+\infty\}$ we denote by
$B_{[i,j]}$ the extension-closure of $\cup_{i\le k \le j}B[k]$.
 We will also use the notation $B_{\ge i} = B_{[i,+\infty]}$ and $B_{\le j} = B_{[-\infty, j]}$.

Then the following statements are valid.

1. For any element $T$ of $B_{[m,n]}$ there exists a distinguished triangle
$T_1 \to T \to T_2 \to T_1[1]$, where $T_1 \in B_{[m,0]}$ and $T_2 \in B_{[1, n]}$.

2. $B_{[i,j]} \subset \cu_{[i,j+1]}$

3. $w$ descends to  $\cu'$, and the heart $\hw'$ is naturally equivalent to $\kar_{\cu'}(\hw[S^{-1}]_{add})$, 
 where $\hw[S^{-1}]_{add}$ is a certain additive localization of the category $\hw$ by $S$ that was defined in \cite{bososn}.
\end{pr}
\begin{proof}
All the assertions were proved in \cite{bososn}: see Proposition 2.1.3(1), Remark 2.1.4(1), and  Theorem 4.2.2 of the paper, respectively.
\end{proof}

\section{On the weight lifting properties}\label{oretwlp}

In the section we will treat  Question \ref{q1} for various localization functors. 

In \S\ref{gentr}  we 
 study extensions in triangulated categories (including so-called weak weight decompositions) and prove some lemmas related to "localizing coefficients" in triangulated categories.

In \S\ref{bcas}  we 
introduce (our versions of) the (left and right) Ore conditions 
 along with weight lifting properties; we also prove a few simple lemmas.

In \S\ref{slwp} we prove our main Theorem \ref{dirsum} (on the left weight lifting property for the localization by $\lan \co(S)\ra$, where $S$ is a set of $\hw$-morphisms satisfying the left Ore condition). We also prove that objects of the heart of $\cu'=\cu/\du$ lift to $\cu_{w=0}$ whenever $\du$ is strongly generated by a subset of $\cu_{w=0}$. 

\subsection{Properties of (weak) weight decompositions and extensions}\label{gentr}


\begin{defi}\label{dwwp}
We call a distinguished triangle $X_1 \to X \to X_2 \to X_1[1]$ a {\it weak weight decomposition} of $X$ if $X_1 \in \cu_{w\le 0}$ and $X_2 \in \cu_{w\ge 0}$.
\end{defi}

\begin{rema}\label{rwwd}
Usually we will consider weak weight decompositions 
in some triangulated subcategory $\du$ of $\cu$ (that contains a given $X$). This means that $X_1$ and $X_2$ are objects of $\du$.
Using this definition, 
Proposition \ref{p2loc}(1) can be reformulated as follows: any element of $\du$ possesses  a weak weight decomposition inside $\du$ (whenever $\du$ is strongly generated by a set of elements of $\cu_{[0,1]}$). 
\end{rema}

\begin{pr}\label{extww}
Let $X,Y \in \obj\cu$ and let $E \in X \star Y$ (see \S\ref{snotata}).
Assume that $X_1 \to X \to X_2 \to X_1[1]$ and $Y_1 \to Y \to Y_2 \to Y_1[1]$ are weight decompositions  (resp., weak weight decompositions) of $X$ and $Y$, respectively.
Then $E$ has a weight decomposition (resp., a weak weight decomposition) $E_1 \to E \to E_2 \to E_1[1]$ such that $E_1 \in X_1 \star Y_1$
 and $E_2 \in X_2 \star Y_2$.
 \end{pr}
\begin{proof}
The orthogonality axiom for $w$ implies that $Y_1[-1]\perp X_2$. Hence Proposition 1.1.9 of \cite{bbd} (cf. also Lemma 1.4.1(1) of \cite{bws}) implies 
the existence of  a commutative diagram
$$\begin{CD}
Y_1[-1]@>{}>> Y[-1]@>{}>> Y_2[-1] \\
@VV{}V@VV{}V@ VV{}V \\
X_1@>{}>> X@>{}>> X_2\\
@. @VV{}V@.\\
@. E@. 
\end{CD}
$$
Applying  Proposition 1.1.11 of \cite{bbd} we deduce the existence of a commutative diagram 
$$\begin{CD}
X_1@>{}>> X@>{}>> X_2\\
@VV{}V@VV{}V@ VV{}V \\
E_1@>{}>> E@>{}>> E_2 \\
@VV{}V@VV{}V@ VV{}V \\
Y_1@>{}>> Y@>{}>> Y_2\end{CD}
$$
whose rows and columns are distinguished triangles.
 Since the classes $\cu_{w\ge 0}$ and $\cu_{w\le 0}$ are extension-closed,
the triangle $E_1 \to E \to E_2 \to E_1[1]$ yields a weight decomposition (resp., a weak weight decomposition) of $E$.
\end{proof}

\begin{coro}\label{cextww}
1.
Let $\du$ be the triangulated subcategory of $\cu$ 
 strongly generated by some $D\subset \obj \cu$.
Then any object of $\du$ possesses a  weight decomposition  (resp.,  a weak weight decomposition) in $\du$ whenever any object of $\bigcup\limits_{i\in \mathbb{Z}} D[i]$ does. 

2. Conversely, if $\du$ is a subcategory of $\cu^b$ 
such that any  object of $\du$ possesses a weight decomposition  (resp., a weak weight decomposition) inside $\du$
then $\du$ is strongly generated by some set of objects of length $0$ (resp., of length $1$).
\end{coro}
\begin{proof}
1.
Suppose $X \in \obj\cu$.
Using induction we can assume that $X$ is an extension of objects of $\bigcup\limits_{i\in \mathbb{Z}} D[i]$; hence it is an extension of objects having weight decompositions (resp., weak weight decompositions).
Then Proposition \ref{extww} yields the result.

2.
It suffices to prove that any $X \in (\obj \du)\cap \cu_{[0,n]}$ belongs to $\lan (\obj \du)\cap\cu_{[0,0]}\ra$ (resp., to $\lan (\obj \du)\cap\cu_{[0,1]}\ra$). 
We will prove the statement using induction on $n$.

The base of induction is 
 the following obvious observation: $X$ belongs to $(\obj \du) \cap \cu_{[0,0]}$ (resp., $X\in (\obj\du)\cap \cu_{[0,1]}$) then $X$ also belongs to $\obj\lan (\obj \du) \cap \cu_{[0,0]}\ra$ (resp., $\obj\lan (\obj\du)\cap \cu_{[0,1]} \ra$).

Now we will describe the inductive step.
Suppose that  any $Y \in (\obj \du)\cap \cu_{[0,n-1]}$ belongs to $\lan (\obj \du)\cap\cu_{[0,0]}\ra$.
Let $X \in (\obj \du)\cap \cu_{[0,n]}$.
Consider a weight decomposition (resp., a weak weight decomposition) of $X$, i.e., a distinguished triangle
$$X_1 \to X \to X_2 \to X_1[1]$$
such that $X_2 \in \cu_{w\ge 1}$ and $X_1 \in \cu_{w \le 0}$ (resp., $X_1\in \cu_{w\le 1}$). 
Since $X_1 \in\obj \lan (\obj \du)\cap\cu_{[0,0]}\ra$ (resp., $X_1 \in\obj \lan (\obj \du)\cap\cu_{[0,1]}\ra$) and
$X_2[-1]$ satisfy the inductive assumption, we obtain that $X \in \obj \lan (\obj \du)\cap\cu_{[0,1]}\ra$.
\end{proof}

 Let R be a commutative ring, and $\cu$ be an $R$-linear triangulated category.  For any set 
  $\Sigma \subset R$ we define $S$ to be the class of endomorphisms of objects of $\cu$ corresponding to 
	multiplications by the elements of $\Sigma$. 
Denote by $K$ the Karoubi-closure of the triangulated subcategory of $\cu$ strongly generated by cones of elements of $S$. 

The following two lemmas treat this simple setting.

\begin{lem}\label{globalexpon}
For any $Y \in \obj K$ there exists an element $\lambda$ which is a product of elements of $\Sigma$ such that $\lambda\id_Y = 0$. Thus we may choose 
 $\lambda \in \Sigma$  if $\Sigma$ is a multiplicative set.  
\end{lem}
\begin{proof}
It suffices to prove the following. 
Let $\lambda_1\id_{Y_1} = 0$ and $\lambda_1\id_{Y_2} = 0$ for some $Y_1,Y_2 \in \obj \cu$.
Then for any distinguished triangle $Y_1 \stackrel{p}\to Y \stackrel{q}\to Y_2 \to Y_1[1]$ the equality $\lambda_1 \lambda_2\id_Y =0$ holds.
Indeed, $q\lambda_2\id_Y = \lambda_2\id_{Y_2} q = 0$, so $\lambda_2\id_Y$ factors through $Y_1$, i.e., $\lambda_2\id_Y = q \circ t$ for some $t \in \cu(Y,Y_1)$. Multiplying the both parts of equality by $\lambda_1$ we obtain the result.
\end{proof}

\begin{lem}\label{endomloc}
For any  $U,X\in \obj \cu$ the following 
conditions are equivalent.

1. There exists a distinguished triangle $U \to X \stackrel{s}\to Y \stackrel{d}\to U[1]$ in $\cu$  such that $Y$ belongs to $K$.

2. There exists a distinguished triangle $Y' \to X \to U \to Y'[1]$ in $\cu$  such that $Y'$ belongs to $K$.
\end{lem}

\begin{proof}
It suffices to prove that condition 1 implies condition 2 since the converse implication is its dual.

Suppose $Y \in \du$.
By Lemma \ref{globalexpon} there exists a product $\lambda$ of elements of $\Sigma$ such that $\lambda \id_Y=0$. 

We note that $d$ factors through $U/\lambda$, where $U/\lambda$ denotes a cone 
 of 
 $\lambda \id_U$. Denote the corresponding morphism from $U/\lambda$ to $U[1]$ by $r$, and denote the morphism from $Y$ to $U/\lambda$ by $q$. 
Certainly, $U/\lambda$ belongs to $K$.
 
 The octahedron axiom applied to the commutative triangle $(q,r,r\circ q)$ (which we specify by its arrows) yields the existence of a distinguished triangle $\co(q)[-1] \to X \to U \to \co(q)$. Note that $\co(q)[-1]$ belongs to $K$ 
since it is a cone of a morphism between objects of $K$. 
So we obtain the result.
\end{proof}


\begin{coro}\label{centestcent}
Let $X,Y, Z$ be subclasses of $\cu$. 

Then the following statements are valid:

1. $X \star (\obj K) = (\obj K) \star X$.

2. $(X \star Y) \star Z = (X\star Y) \star Z$
\end{coro}
\begin{proof}
The first assertion is immediate from Lemma \ref{endomloc}. 
 The second assertion is given by Lemma 1.3.10 of \cite{bbd}. 
\end{proof}

Now assume $\cu$ is endowed with a weight structure $w$. Denote by $K'$ the 
 triangulated subcategory of $\cu$ strongly generated by the class $\langle \co(X \stackrel{\lambda}\to X) : X \in \obj \hw, \lambda \in \Sigma \rangle$. 

\begin{lem}\label{wwstor}
If $w$ is bounded and $\cu$ is small then $K$ coincides with $\kar_{\cu}(K')$.\footnote{The smallness of $\cu$ actually seems to be unnecessary. Yet this restriction (which ensures the existence of "all" Verdier localizations of $\cu$) 
 is quite satisfactory for our purposes.}
\end{lem}
\begin{proof}
Let $U$ be an object of $\cu$. We should show that $U/\lambda = \co(U \stackrel{\lambda}\to U)$ belongs to $\kar_{\cu}(K')$. 
The morphism $t(L(\lambda\id_U))$ is invertible in $K_w(\hw[\Sigma^{-1}]_{add})$, where $L$ denotes the localization functor $\cu \to \cu/K'$ and $t$ is the weight complex functor $\cu/K' \to K_w(\hw[\Sigma^{-1}]_{add})$ (see section 3 of \cite{bws} for the definition). Now the "conservativity" part of Theorem 3.3.1 of \cite{bws} implies that $L(U/\lambda)$ is isomorphic to zero in $\cu/K'$. Hence $U/\lambda$ belongs to the Karoubi-closure of $K'$.
\end{proof}

\subsection{The Ore conditions and  weight lifting properties}\label{bcas}

\begin{defi}\label{strore}
A class of morphisms $S \subset \mo(\au)$ (for a category $\au$) is said to satisfy the 
{\it left Ore condition}  if  for any couple of morphisms $g\in \au (X, Y)\cap S$ and $u \in \au(X, Z)$ 
(for some $X,Y,Z\in \obj \au$) can be completed to a commutative square
\begin{equation}\label{redf}
\begin{CD}
X    @>{g}>>  Y \\ 
@VV{u}V       @VV{f}V\\
Z    @>{s}>>   T
\end{CD}
\end{equation}
such that $s$ belongs to $S$ also.

The right Ore condition 
 is the categorical dual to the left Ore condition. 






Note that we do not include the usual "cancellation" property into our Ore conditions (so, one may say that we are only interested in one of the two Ore conditions).
\end{defi}

Till the end of the section we will always assume that 
 $w$ is a weight structure on a triangulated category $\cu$, $S \subset \mo(\hw)$ is a set of morphisms, $\du \subset \cu$ is the triangulated subcategory of $\cu$ strongly generated by the set $B$ of cones of   elements of $S$  (unless stated otherwise).
According to Proposition \ref{p2loc},  $w$ descends to the localization $\cu'=\cu/\du$; we will use the notation $w'$ for the corresponding descended weight structure.

\begin{lem}\label{loreg}
Suppose that $S$ satisfies the left Ore condition.
Then for any $s \in S, X \in \cu_{w\ge 0}$, and $u\in \cu(\co(s)[-1],X)$ there exists an element $X'$ of $\cu_{w\ge 0}$ along with a morphism $f\in \cu(X,X')$ such that $\co(f) \in \obj\du \cap \cu_{[0,1]}$ and the composition morphism $f\circ u$ from $\co(s)$ to $X'$ is zero.
\end{lem}
\begin{proof}
Let $w_{\le 0}X \stackrel{i}\to X \to w_{\ge 1}X$ be a weight decomposition of $X$.
The morphism $u$ factors through $w_{\le 0}X$ by the orthogonality axiom for $w$. Denote the corresponding map from $\co(s)[-1]$ to $w_{\le 0}X$ by $u'$.
The left Ore condition yields that there exist $Y \in \cu_{w=0}$ and  $g\in \cu(w_{\ge 0}X,Y)$ such that  $g \circ u=0$ and $\co(g) \in \obj\du$. Denote by $v$ the corresponding morphism from $\co(g)[-1]$ to $w_{\le 0}X$.

Applying the octahedron axiom to the commutative triangle $(v,u, u\circ v)$ we obtain a morphism $g$ from $X$ to some $X'$ whose cone belongs to $ (\obj \du) \cap \cu_{[0,1]}$ as demonstrated by the 
following diagram:
\begin{equation}\label{oreg}
\begin{CD}
\co(g)[-1]    @>{}>>  \co(g)[-1]  \\ 
 @VV{v}V       @VV{}V            \\
w_{\le 0}X    @>{}>>  X @>{}>> w_{\ge 1}X  \\ 
@VV{g}V       @VV{}V      @VV{}V\\
Y    @>{}>>  X' @>{}>> w_{\ge 1}X\\
\end{CD}
\end{equation}
Certainly, the corresponding morphism from $\co(s)[-1]$ to $X'$ is zero since it factors through 
$g\circ u'=0$. 

\end{proof}

\begin{lem}\label{twprop}
Let $X$ be an object of $\cu$ and let $w_{\le -1}X \stackrel{i}\to X \to w_{\ge 0}X$ be a shifted weight decomposition of $X$.
Let $f$ be a morphism from $X$ to $X'$ such that $\co(f)= E \in (\obj \du) \cap \cu_{[0,n]}$ and that  $f\circ i$ factors through an object $T \in (\obj \du) \cap \cu_{[0,n]}$ for some $n \in \mathbb{N} \cup \{+\infty\}$.

Then there exists a morphism from $X'$ to some $X'' \in \cu_{w \ge 0}$ whose cone belongs to $(\obj \du) \cap \cu_{[-1,n-1]}$.
\end{lem}
\begin{proof}

Certainly, we have the following commutative square.

\begin{equation}\label{orreg}
\begin{CD}
w_{\le -1}X    @>{}>>  T  \\ 
 @VV{}V       @VV{}V            \\
w_{\le -1}X'    @>{}>>  X'  
\end{CD}
\end{equation}

By 
Proposition 1.1.11 of \cite{bbd}, it can be 
completed to the following diagram (whose rows and columns are distinguished triangles):
\begin{equation}\label{orrreg}
\begin{CD}
w_{\le -1}X    @>{}>>  T      @>{}>>  C_1\\ 
 @VV{}V       @VV{}V      @VV{}V            \\
w_{\le -1}X'    @>{}>>  X'      @>{}>>   w_{\ge 0}X'\\ 
 @VV{}V       @VV{u_1}V            @VV{}V      \\
E         @>{}>>  C_2      @>{u_2}>>    X''
\end{CD}
\end{equation}

The extension-closedness of $\cu_{[0,n]}$ yields that $C_1$ belongs to $\cu_{[0,n]}$; hence 
$X''$ belongs to $\cu_{w\ge 0}$.
Moreover, $T[1] = \co(u_1)$ belongs to $(\obj \du) \cap \cu_{[-1,n-1]}$, and
$E[1] = \co(u_2)$ also does.
So the composition $u_2\circ u_1$ 
satisfies the conditions desired.
\end{proof}

\begin{defi}\label{swlp}
Let $\du$ be a full triangulated subcategory of $\cu$; denote by $L$ the localization functor $\cu\to \cu'=\cu/\du$.
 
Then we will say that the triple $(\cu,w,\du)$ satisfies the  {\it left weight lifting property} if $w$ descends to a weight structure $w'$ on $\cu$ (see Remark \ref{rdesc})  and for any object $X$ of $\cu$ such that $F(X) \in \cu'_{w'\ge 0}$ 
there exists a distinguished triangle $$T \to X \to X_{\ge 0} \to T[1]$$
such that  $T\in \obj \du $ and $X_{\ge 0} \in \cu_{w\ge 0}$.


We will say that $F$ satisfies the {\it right weight lifting property} if the corresponding functor between the opposite triangulated categories satisfies the  left weight lifting property.

\end{defi}


The left weight lifting property can be slightly reformulated if there exist  weak weight decompositions inside $\du$.

\begin{lem}\label{lwp}

1. Assume that $(\cu,w,\du)$ satisfies the  left weight lifting property and that  for any object of $\du$ there exists its $w$-weak weight decomposition inside $\du$.
 Then for any $X\in \obj \cu$ such that $L(X)\in \cu'_{w'\le 0}$ there exists a distinguished triangle $T' \to X \to X'_{\ge 0} \to T[1]$
such that  $T'\in (\obj \du)\cap \cu_{w\le 0}$ and $X'_{\ge 0} \in \cu_{w\ge 0}$.

2. In particular, this statement is fulfilled whenever $(\cu,w,\du)$ satisfies the  left weight lifting property and $\du$ is strongly generated by 
some elements of $\cu_{[0,1]}$.

\end{lem}
\begin{proof}
1. The  left weight lifting property gives the existence of a distinguished triangle $T \to X \to X_{\ge 0} \to T[1]$ with $T\in \obj \du $ and 
 $X_{\ge 0} \in \cu_{w\ge 0}$. Now consider a weak weight decomposition $T_1 \to T \to T_2 \to T_1[1]$ of $T$ inside $\du$ (recall that $T_1\in \obj \du\cap \cu_{w\le 0}$ and $T_2\in \obj \du\cap \cu_{w\ge 0}$).
These triangles yield that $X\in \{T\}\star  \cu_{w\ge 0}$ and $T\in (\obj \du\cap \cu_{w\le 0})\star \cu_{w\ge 0}$. Thus $X\in ( (\obj \du\cap \cu_{w\le 0})\star \cu_{w\ge 0})\star  \cu_{w\ge 0}$. The latter class equals  $(\obj \du\cap \cu_{w\le 0})\star (\cu_{w\ge 0}\star  \cu_{w\ge 0})$ according to Corollary \ref{centestcent}(2). Since $\cu_{w\ge 0}\star  \cu_{w\ge 0}=\cu_{w\ge 0}$ (see Proposition \ref{pbw}(\ref{iort})), we obtain the result.



2. Recall that objects of $\du$ possess  $w$-weak weight decompositions inside $\du$ whenever  $\du$ is strongly generated by some elements of $\cu_{[0,1]}$ (see  Remark \ref{rwwd}).

\end{proof}  

\subsection{The main weight lifting results}\label{slwp}

\begin{theo}\label{dirsum}
Suppose that 
 a set of morphisms $S$ of the category $\hw$ satisfies the left Ore condition; take $\du=\lan \co(S)\ra$.
Then $(\cu,w,\du)$ satisfies the  left weight lifting property.

\end{theo}
\begin{proof}
Let $X$ be an object of $\cu$ such that $L(X) \in (\cu')_{w'\ge 0}$. It suffices to prove that there exists an element  $X' \in \cu_{w\ge 0}$ along with a morphism $f \in \cu(X,X')$ such that $\co(f)[-1] \in (\obj \du) \cap \cu_{w\le 0}$. 

Assume first that 
 $X$ is bounded below. Then we'll use induction on the minimal $n\in \z$ such that $X$ belongs to $\cu_{w\ge -n}$.
Fix  a (shifted) weight decomposition of $X$: $w_{\le-1}X\stackrel{i}\to X \stackrel{j}\to w_{\ge 0}X \to w_{\le-1}X[1]$.

Suppose  $X\in \cu_{w\ge -1}$. 
The orthogonality axiom for the weight structure  $w'$ yields $L(i) = 0$ (since $L(X) \in \cu'_{w'\ge 0}$).
By 
Lemma 2.1.26 of \cite{neebook}, there exists a factorization of $i$ through some object $M$ of $\du$, i.e., there
exist $g \in \cu(w_{\le -1}X,M)$ and $f \in \cu(M, X)$ such that $f\circ g = i$.
According to Proposition 2.1.3(3) of \cite{bososn} we can assume that $M \in (\obj \du) \cap \cu_{[-1,1]}$.

Let $M_1 \stackrel{s_1}\to M \stackrel{s_2}\to M_2$ be a weak weight decomposition of $M$ in $\du$ 
(see Remark \ref{rwwd}). 
By Lemma \ref{loreg}, there exists a morphism $r\in \cu(X,X')$ such that $r \circ f \circ s_1 = 0$ and $\co(r) \in (\obj \du)\cap \cu_{[-1,0]}$.
Hence $r \circ f$ factors through $M_2$. 

By Lemma \ref{twprop}, we have a morphism $u$ from $X'$ to $X''$ such that $X'' \in \cu_{w \ge 0}$ and $\co(u) \in (\obj   \du)\cap \cu_{[-1,0]}$.
Thus the composition $u \circ r$ yields the morphism desired from $X$ to an object whose cone belongs to $(\obj  \du)\cap \cu_{[-1,0]}$.

Now we describe the inductive step. 
If $X \in \cu_{w\ge -n}$ then one can apply the inductive assumption to $X[n-1]$. It yields a distinguished
triangle of the form $M \to X \stackrel{p}\to X'\to M[1]$, where $M \in (\obj \du) \cap \cu_{w\le -(n-1)}$ and $X' \in \cu_{w\ge -(n-1)}$. By the inductive assumption, there also exists a distinguished triangle $M' \to X' \stackrel{q}\to X'' \to
M'[1]$, where $M' \in (\obj \du) \cap \cu_{w\le 0}$ and $X'' \in \cu_{w\ge 0}$. Applying the octahedron axiom to the commutative triangle of morphisms
$(p,q, q\circ p)$ we obtain a triangle of the type desired for $X$.

It remains to 
 prove that any 
 $X$ satisfying $L(X) \in (\cu')_{w'\ge 0}$ is actually necessarily bounded below.  All objects of $\du$ are $w$-bounded below (since $\du$ 
 is strongly generated by objects satisfying this property).  
The morphism $w_{\le -1} X \stackrel{i}\to X$ becomes zero in $\cu'$ and hence factors through an object 
 $T \in (\obj \du) \cap \cu_{w\ge N}$ for some  $N\in \z$. 
Hence the morphism  $w_{\le N-1} X \stackrel{i_n}\to X$ factors through $T$ also. 
 Since $w_{\le N-1}X\perp T\in \cu_{w\ge N}$, $i_n$ is zero;  thus $X \in \cu_{w\ge N}$.


\end{proof}

\begin{rema}\label{restrcase}
1. Now we describe an important particular case of our theorem.

Suppose that  $S$ consists of morphisms of the form $A \to 0$ for 
$A$ running through some set $D\subset \cu_{w=0}$. Then $w$ restricts to $\du=\lan D\ra$  (in the sense of Definition \ref{dwso}\ref{iVII}; the corresponding weight structure will be denoted by $w_{\du}$) immediately from Corollary \ref{cextww}, and $S$  certainly satisfies the left Ore condition. 

Applying Theorem \ref{dirsum} (and Lemma \ref{lwp}(2)) to this setting we obtain the following: for any object $X \in \obj \cu$ such that $L(X) \in \cu'_{w'\ge 0}$  there exist $X' \in \cu_{w\ge 0}$ and  $g\in \cu(X,X')$ such that $\co(g)[-1] \in (\obj\du) \cap \cu_{w\le 0} = \du_{w_{\du}\le 0}$.
Denote the corresponding morphism from $\co(g)[-1]\to X$ by $j$.
Note that $\co(g)[-1]$ admits a shifted weight decomposition $C_1 \stackrel{i}\to \co(g)[-1] \to C_2$ such that $C_1 \in \du_{w_{\du}\le -1}$, $C_2 \in \du_{w_{\du}\ge 0}$. Applying the octahedron axiom to the commutative triangle $(i,j, j\circ i)$ 
 we obtain that $X$ admits a shifted $\cu$-weight decomposition $X_1 \to X \to X_2$ such that $X_1$ 
belongs to $\du_{w_{\du}\le -1}$ and $X_2$ belongs to $\cu_{w\ge 0}$.

2. It appears that localizing $\cu$ by $\du=\lan \co(S)\ra$ is only 
"relevant"  in the  case of a bounded $w$. In the case where $\cu$ is closed under coproducts one usually wants to localize it by subcategories that are closed with respect to coproducts also (cf. \S4.3.1 of \cite{bososn} or \S3 of \cite{bsnew}), whereas the localization by any non-zero $\du=\lan \co(S)\ra$ yields a degenerate weight structure.

\end{rema}


We easily deduce the following statement.

\begin{pr}\label{heartlift}
1. Suppose 
that $\du$ is strongly generated by a class $D$ of objects of $\hw$.
Then $w$ descends to $\cu'=\cu/\du$, and $\cu'_{w'=0}$ 
equals the essential image $L(\cu_{w=0})$.

2. Suppose that $\du$ is strongly generated by $\co(S)$, where $S$ is a class of morphisms between objects of $\hw$ satisfying  both the left and the right Ore conditions.
Then 
 $\cu'_{w'=0}$  
lies in (the essential image) $L(\cu_{[-1,0]}) $. 
\end{pr}
\begin{proof}
1. Let $S$ be the set $\{0\to A:\ A\in D \}$. 
The category $\du$ is strongly generated by cones of elements of $S$ (cf. Remark \ref{restrcase}(1)). Hence $w$ descends to $\cu'$; thus $L(\cu_{w=0})$ lies in $\cu'_{w'=0}$.

Now we verify the converse inclusion. 
According to the dual to Theorem \ref{dirsum}, any $Y \in \cu'_{w'=0}$ can be lifted to an object $X$ of $\cu_{w\le 0}$. 
According to Remark \ref{restrcase}(1), there exists a shifted weight decomposition $w_{\le -1} X \to X \stackrel{j}\to w_{\ge 0} X$ such that $w_{\le -1} X$ belongs to $\obj\du$. 
Hence $L(w_{\ge 0} X)\cong Y$. 
Since $w_{\ge 0} X \in \cu_{w=0}$, we obtain the result.

2. 
Applying the dual to Theorem \ref{dirsum} we obtain that
any $Y \in \cu'_{w'=0}$ can be lifted to an object $X$ of $\cu_{w\le 0}$.
Applying Theorem \ref{dirsum} to $X[1]$ we obtain a distinguished triangle $
 X \to X_{\ge -1} \to T[1]\to X[1]$ 
such that  $X_{\ge -1}\in \cu_{w\ge -1}$ and $T\in \obj \du$. 
 Moreover, Lemma \ref{lwp}(2) allows us to assume that $T\in \obj \du\cap \cu_{w\le -1}$. 
Hence $X_{\ge -1}$ is an extension of objects of $\cu_{w\le 0}$; thus we have $X_{\ge -1} \in \cu_{[-1,0]}$.
By construction $L(X_{\ge -1}) \cong L(X) \cong Y$ 
and we obtain the result.
\end{proof}

\begin{rema}
Certainly, Proposition \ref{heartlift}(1)  implies Theorem 1.7(b) of \cite{wildint}. 
Moreover, 
 this result easily extends to motives with arbitrary coefficients. 
Note that we do not use any extra properties of our localization setting in this proof (such as the existence of an adjoint functor to the localization one).

\end{rema}

\section{Applications}\label{sappl}

In \S\ref{spectate} we give  concrete examples 
for Theorem \ref{dirsum}; we study  the stable homotopy category and the (triangulated) category of Tate motives. 

In \S\ref{apppl} we compare the "refined" weights of  $X\in \obj \cu$ in a system of localizations $\cu/\du_i$ with 
 its weights in the categories $(\cu/\du_i) \otimes \z[J^{-1}]$ (where $\du_i$ are some triangulated subcategories to which $w$ restricts) for some set of integers $J$. 
In Proposition \ref{torsindep} we show that there exists  a distinguished triangle $T\to X \to X' \to T[1]$, where $T$ is a torsion object and the weights of $X'$ in all $\cu/\du_i$ are almost the   same as the weights of $X$ in $(\cu/\du_i)\otimes \z[J^{-1}]$. 

\subsection{On finite spectra and Tate motives}\label{spectate}

Now we are able to prove the results stated in the beginning of the introduction.

 
\begin{rema}\label{tappl}
1. Let $SH^{fin}$ be the topological stable homotopy category of finite spectra. As 
shown in \S4.6 of \cite{bws}, there exists a certain {\it spherical} weight structure on this category 
 such that $SH^{fin}_{w\ge n+1}$ consists of finite spectra $X$ such that the stable homotopy groups $\pi^s_i(X)$ are zero for $i\le n$.

Let $S$ be the class of multiplications by non-zero integers 
on objects of $\hw$ (that consists of finite coproducts of sphere spectra). Consider the localization functor $SH^{fin} \stackrel{L}\to SH^{fin}/\langle\co(S)\rangle$. 
The category $SH^{fin}/\langle\co(S)\rangle$ is isomorphic to $SH^{fin} \otimes \q$ (see  Appendix A.2 of  \cite{kellymotcharp} or Proposition 5.6.2 of \cite{bpure}; note  also that  $SH^{fin} \otimes \q$ is equivalent to $K^b(\q-\operatorname{vect})$, where $\q-\operatorname{vect}$ is the category of finite-dimensional $\q$-vector spaces); so that for $Y\in \obj SH^{fin}$ we will denote $L(Y)$ by $Y_{\q}$.
Denote by $w_{\q}$ the weight structure on the category $SH^{fin} \otimes \q$ 
 that descends from the weight structure $w$ on $SH^{fin}$.  Certainly, $Y_{\q}$ belongs to 
$(SH^{fin} \otimes \q)_{w_{\q}\le -n-1}$ if and only if the groups $SH^{fin} \otimes \q (\mathbb{S}_{\q}^i, Y_{\q}) = SH^{fin}(\mathbb{S}^i, Y) \otimes \q \cong \pi^s_{i}(Y)\otimes \q$ vanish for $i\le n$. 

Let $X$ be an object of $SH^{fin}$ such that the stable homotopy groups $\pi^s_{i}(X)$ are torsion for all $i \le n$. This means $X_{\q}$ belongs to 
$SH^{fin} \otimes \q_{w_{\q}\le -n -1}$. Applying the dual of Theorem \ref{dirsum} to this setting we obtain the following: there exists a distinguished  
triangle $\tilde{X} \to X \to T \to \tilde{X}[1]$ in $SH^{fin}$ such that all the homotopy groups of $T$ are torsion and $\tilde{X}$ is an $n$-connected spectrum. 

Note that the natural candidate for $\tilde{X}$, the $n$-connected cover, usually does not belong to $SH^{fin}$ (for example, a $1$-connected cover of $\mathbb{S}/2$ is an extension of $\mathbb{S}/2$ by a shift of the Eilenberg-Maclane space, $H\z/2$). 

2. Let $F$ be a field
, $R$ be either $\z$ or $\q$. 
Consider the category  $DMT^{eff}_{gm}(F;R)$ of  effective geometric $R$-linear Tate motives over  $F$.

Recall that there is a weight structure $w$ on this category whose heart is the 
 additive hull of $\{R(k)[2k]\}_{k \in \mathbb{N}}$ (see Theorem 2.2.1 of \cite{bzp}; see also 
\cite{canc} or  \cite[Theorem 2.1.2]{bokum} to avoid $\z[1/p]$-coefficients in the case $F$ is of positive characteristic $p$). 
Certainly it restricts to the full triangulated subcategory of $1$-effective Tate motives. It follows that the weight structure descends 
to the category $DMT_{gm}^0(F) = DMT^{eff}_{gm}(F;R)/DMT^{eff}_{gm}(F;R)(1)$. 
Since the subcategory $DMT^{eff}_{gm}(F;R)(1)$ is orthogonal to $R$, it is also orthogonal to $R/m = \co(R \stackrel{m}\to R)$. By Lemma 9.1.5 of \cite{neebook}, 
$H^n_M(M,R/m)= DMT^{eff}_{gm}(F;R)(M,R/m[n]) \cong DMT_{gm}^0(F;R)(L(M), L(R/m)[n])$. Note that $L(R)$ generates the heart of $w^0$. Moreover, $DMT_{gm}^0(F;R)(L(R), L(R)[n]) \cong 0$ whenever $n$ is not zero. Thus, the category $DMT_{gm}^0(F;R)$ is equivalent to the category $K^b(\operatorname{Free}(R))$ (where $\operatorname{Free}(R)$ is the category of finite dimensional free modules over $R$).

Now let $M$ be an object of $DMT^{eff}_{gm}(F;R)$. Assume that the motivic cohomology groups $H^i_M(M,\z/m)$ are zero for $i\ge -n, m\in \z$. This can be rewritten as  $K^b(\operatorname{Free}(R))(L(M),L(R/m)[i]) \cong DMT_{gm}^0(F,R)(L(M), L(R/m)[i]) \cong 0$ for $i \ge -n, m\in \z$, i.e., $L(M)$ viewed as a complex (i.e., as an object of $K^b(\operatorname{Free}(R))$) is concentrated in degrees less than or equal to $-(n+1)$. Thus $L(M)$ belongs to $DMT_{gm}^0(F;R)_{w^0\ge n+1}$.

Applying Theorem \ref{dirsum} we obtain the existence of a distinguished triangle 
$E \to M \to M' \to E[1]$ such that $E$ is an effective motive and $M'$ is an element of $DMT^{eff}_{gm}(F;R)_{w\ge n+1}$.

Note that we can consider vanishing of the groups $H^n_M(-,\q)$ (resp., of $H^n_M(-,\q/\z)$) instead of vanishing of $H^n_M(-,R/m)$ for all $m\in \z$ if $R = \q$ (resp., $R= \z$).
\end{rema}


\subsection{On chains of localizations}\label{apppl}

Using Theorem \ref{dirsum} (and its particular case described in Remark \ref{restrcase}) we make the following observation. 
It 
is applied in  \cite{bosocwh}  to the study of so-called Chow-weight homology theories and related questions.

We assume that $w$ is a bounded weight structure on a small triangulated category $\cu$. Consider a system of full triangulated subcategories $\du_i$ of $\cu$ (for $i\ge 0$; we allow $i$ to be equal to $+\infty$) such that $w$ restricts to $\du_i$, $\du_{i+1} \subset \du_i$, $\du_0 = \cu$ and $\du_{+\infty} = \{0\}$. 
Let $J\subset \z\setminus \ns$; 
let $S$ be the class of morphisms $X \stackrel{p\id_X}\to X$, where $p\in J$, $X\in \obj \hw$. Denote by $K$ the minimal Karoubi-closed full triangulated subcategory of $\cu$ containing $\co(S)$. Denote by $\du'_i$ the category $K \star \du_i$; note that it is  triangulated by Lemma \ref{endomloc} (combined with Lemma \ref{wwstor}).

Let $a_i$ be a non-decreasing non-negative integer sequence (we allow $a_i$ to be equal to $+\infty$). 
Denote by $L^i$ (resp., by $L^{'i}$) the corresponding localization functor from $\cu$ to $\cu/\du_i$ (resp., to $\cu/\du'_i$).

\begin{pr}\label{torsindep}
Let $X$ be an object of $\cu$ such that $L^{'a_i}(X) \in (\cu/\du'_{a_i})_{w\ge -i+1}$ for all $i \in \mathbb{Z}$. 

1. Assume $a_l =0$ for some $l$. 
Then for any integer $N$ there exists a distinguished triangle $T\to X\to M \to T[1]$ such that $T \in \obj K$ and  $M$ is an extension of  an element $M'\in \cu_{w\ge -N+1}$ such that  $L^{a_i}(M') \in (\cu/\du_{a_i})_{w\ge -i+1}$  for any $i\in \z$ by an element of $(\obj \du_{a_N})\cap \cu_{w\le -N}$. 

2. For any integers $N$ and $N'$ there exists a distinguished triangle $T\to X\to M \to T[1]$ such that $T \in \obj K$ and  $M$ belongs to the class $((\obj \du_{a_N})\cap \cu_{w\le -N}) \star \{M'\} \star \cu_{w\ge -N'}$,  where $M'$ is an element of $\cu_{w\ge -N+1}$ such that  $L^{a_i}(M') \in (\cu/\du_{a_i})_{w\ge -i+1}$ for all $i\in \z$. 
\end{pr}
\begin{proof}
Combining  Theorem \ref{dirsum} with  Lemma \ref{lwp}(2)   and Proposition \ref{extww} we obtain that $X \in ((\obj K) \star (\obj  \du_{a_n})) \star \cu_{w\ge -n+1} = ((\obj K) \star ((\obj \du_{a_n})\cap \cu_{w\le -n})) \star \cu_{w \ge -n+1}$. 
By Corollary \ref{centestcent}(2), $X$ belongs to $(\obj K) \star (((\obj \du_{a_n})\cap \cu_{w\le -n}) \star \cu_{w \ge -n+1}) = (\obj K) \star (((\obj \du_{a_n})\cap \cu_{w\le -n}) \star \cu_{w \ge -n+1})$.

Hence for any integer $n$ there exists a distinguished triangle $T_0 \to X \to X' \to T_0[1]$ such that $T_0 \in \obj K$ and $X'$ is an extension of an element $X''$ of $\cu_{w\ge -n+1}$ by an element $Y_0$ of $(\obj \du_{a_n})\cap \cu_{w\le -n}$. 
We will use this statement later in the proof.

Now we prove the first assertion.

1. 
Denote by $n$ the maximal integer such that $a_n = 0$.
We will prove the statement using induction on $N - n$ for all triangulated categories $\cu$ with a weight structure and a filtration by full subcategories $\du_i$, to which $w$ restricts. 

In the base case $n\ge N$. 
As we have noticed in the beginning of the proof there exists a distinguished triangle $T_0 \to X \to X' \to T_0[1]$ such that $X'$ is an extension of an element $X''$ of $\cu_{w\ge -N+1}$ by an element $Y_0$ of $(\obj \du_{a_N})\cap \cu_{w\le -N}$. Note that $L^{a_i}(X'') \in (\cu/\du_{a_i})_{w\ge -i+1}$ for all $i\ge n$ by exactness of the localization functor and $L^{a_i}(X'') = L^{0}(X'') \in (\cu/\du_0)_{w\ge -i+1}$ for $i<n$ since $\du_0 = \cu$.

Now we describe the inductive step. Let us assume that the statement is true for all triangulated categories with a weight structure 
  equipped with a filtration by full subcategories and $a_k$ being a sequence such that $a_{n}=0$ and $a_k >0$ for $k> n$ for some $n\ge l$. We want to prove the statement for a triangulated category $\cu$ with a filtration by $\du_i$ and $a_k$ being a sequence such that $a_{l-1} =0$ and $a_k >0$ for $k>l-1$.

Let $X$ be an object of $ \cu$ be such that $L^{'a_i}(X) \in (\cu/\du'_{a_i})_{w\ge -i+1}$ for any $i \in \mathbb{Z}$. 
The argument  in the beginning of the proof shows that there exists a distinguished triangle $T_0 \to X \to X' \to T_0[1]$ such that $T_0 \in \obj K$ and $X'$ is an extension of an element $X''$ of $\cu_{w\ge -l+1}$ by an element $Y_0$ of $(\obj \du_{a_l})\cap \cu_{w\le -l}$.

Certainly, $L^{'a_j}(X') \in (\cu/\du'_{a_j})_{w\ge -j+1}$ for any integer $j\in\z$. Moreover, 
$L^{'a_j}(X''[-1]) \in (\cu/\du'_{a_j})_{w\ge -j+1}$ for all $j> l$ by the weight-exactness of the corresponding localization functors and $L^{'a_j}(X''[-1]) \in (\cu/\du'_{a_j})_{w\ge -j+1}$ for all $j\le l$ since $L^{'a_j}(Y_0)=0$  for any $j \le l$. 
By the extension-closedness of the classes $(\cu/\du'_{a_j})_{w\ge -j+1}$ we have $L^{'a_j}(Y_0) \in (\cu/\du'_{a_j})_{w\ge -j+1}$. 

Now we apply the inductive assumption to the object $Y_0$ in $\du_{a_l}$, where the filtration on $\du_{a_l}$ is the one induced from $\cu$ and $a_{i,\du_{a_l}} = \max\{a_i-a_l,0\}$. Obviously, $a_{l,\du_{a_l}} = 0$; hence  the maximal integer $i$ such that $a_{i,\du_{a_l}} = 0$ is bigger or equal than $l$, and we can apply the inductive assumption. 
We obtain that $Y_0$ belongs to the class $(\obj K) \star (Q \star Y)$, where  $L^{a_{j,\du_l}}(Y) \in (\du_{a_l}/\du_{a_j})_{w\ge -j+1}$ for $j \in \z$, $Y\in\cu_{w\ge -N+1}\cap (\obj\du_l)$ and $Q \in (\obj \du_{a_N})\cap \cu_{w\le -N}$. 
Thus $X \in (\obj K) \star (((\obj K) \star (Q \star Y)) \star X'')$. 
By Corollary \ref{centestcent}(2), $X$ also belongs to $(\obj K) \star (Q \star (Y \star X''))$.  
Recall what it means: there exists a distinguished triangle $T \to X \to M \to T[1]$ such that $M$ is an extension of $M' \in Y \star X''$ by $Q \in (\obj \du_{a_N})\cap \cu_{w\le -N}$. 
Note that $L^{a_j}(Y)$ belongs to $(\cu/\du_{a_i})_{w\ge -i+1}$ for $i> l$ by the definition of $a_{i,\du_l}$ and $L^{a_j}(Y)$ belongs to $(\cu/\du_{a_i})_{w\ge -i+1}$ for $j\le l$ since $Y\in \obj\du_l$.  
Since $L^{a_i}(Y)$ and $ L^{a_i}(X'')$  belong to $(\cu/\du_{a_i})_{w\ge -i+1}$ for $i \in \z$ and $X'',Y$ belong to $\cu_{w\ge -N+1}$, 
we obtain the desired properties of the triangle.

2.  
As shown in the beginning of this proof, there  exists a distinguished triangle $T_0 \to X \to X' \to T_0[1]$ such that $T_0 \in \obj K$ and $X'$ is an extension of an element $X''$ of $\cu_{w\ge -N'}$ by an element $Y_0$ of $(\obj \du_{a_{N'+1}})\cap \cu_{w\le -N'-1}$. 
On the one hand $Y_0$ belongs to $\obj \du_{a_{N'+1}}$, so $L^{'a_i}(Y_0)=0$ belongs to $(\cu/\du'_{a_i})_{w\ge -i+1}$ for any $i \le N'+1$. On the other hand $L^{'a_i}(Y_0) \in (\cu/\du'_{a_i})_{w\ge -i+1}$ as it is an extension of $L^{'a_i}(X')\in (\cu/\du'_{a_i})_{w\ge -i+1}$ by $L^{'a_i}(X''[-1]) \in (\cu/\du'_{a_i})_{w\ge -N'+1}$. 
Applying the previous assertion to the object $Y_0$ in $\du_{a_{N'+1}}$ (where the filtration is the one induced from $\cu$ and $a_{i,\du_{a_{N'+1}}} = a_i - a_{N'+1}$) we obtain that $Y_0 \in (\obj K) \star (Q\star M')$, where $L^{a_i}(M') \in (\cu/\du_{a_i})_{w\ge -i+1}$ for any $i\in \z$ and $Q \in (\obj \du_{a_N})\cap \cu_{w\le -N}$.  

We obtain that $X$ belongs to $(\obj K) \star ((\obj K) \star (Q\star M')) \star \{X''\}$. By Corollary \ref{centestcent}(2), there exists a distinguished triangle $T \to X \to M \to T[1]$ such that $T$ is torsion and $M$ belongs to $((\obj \du_{a_N})\cap \cu_{w\le -N})\star \{M'\} \star \cu_{w\ge -N'}$, where $L^{a_i}(M') \in (\cu/\du_{a_i})_{w\ge -i+1}$ for any $i\in \z$.
\end{proof}


\begin{rema}\label{indc}
The proof above wouldn't work if we only knew that every element of $(\cu/\du'_{a_i})_{w\ge 0}$ is a retract of an element of $L^{'a_i}(\cu_{w\ge 0})$.


Indeed, let $X$ be an element of $\cu_{w\ge n}$ such that $L^{'a_{n+1}}(X)$ belongs to $(\cu/\du'_{a_{n+1}})_{w\ge n+1}$. 
To make the inductive step we should be able to present $X$ as an element of the extension-closure of the class $(\obj \du_{a_{n+1}}) \cup (\obj K) \cup \{X'\}$, where $X'$ is some object of smaller length than $X$ (in our proof we have $X' \in \cu_{w\ge n})$. However, we can do this for arbitrary $X$ if and only if the localization functor $L^{'a_{n+1}}$ satisfies the weight lifting property. 

\end{rema}

\begin{rema}\label{tateappl}
1. If $\cu$ is $R$-linear, then  the conclusion of Proposition \ref{torsindep} certainly remains valid if we take $J \subset R$.

2. If the sequence $a_i$ is integral and stabilizes, i.e. there exists $n$ such that $a_i =a_{i+1}\in \n$ for $i\ge n$ (equivalently, $a_i$ is bounded above) then Proposition \ref{torsindep}(1) yields the following:

For any $M$ such that  $L^{'a_i}(M) \in (\cu/\du'_{a_i})_{w\ge -i+1}$  there exists a distinguished triangle $T\to X\to M \to T[1]$ such that $T \in \obj K$ and $L^{a_i}(M) \in (\cu/\du_{a_i})_{w\ge -i+1}$. 

Indeed, by Proposition \ref{torsindep}(1) we have a distinguished triangle $T\to X\to M \to T[1]$ such that $T \in \obj K$ and $M$ is an extension of $M'$ such that $L^{a_i}(M') \in (\cu/\du_{a_i})_{w\ge -i+1}$ by an element $Q$ of $(\obj \du_{a_n})\cap \cu_{w\ge -n+1}$. 
Since $a_i\le n$, $L^{a_i}(Q) = 0$ and thus $L^{a_i}(M) \in (\cu/\du_{a_i})_{w\ge -i+1}$.

3. Adopt the notation of Remark \ref{tappl}(2) (so, $\cu = DMT^{eff}_{gm}(F) = DMT^{eff}_{gm}(F;\z)$ and $\du = DMT^{eff}_{gm}(F)(1)$); let $J$ be the set of  non-zero integers. 

Let $X$ be an object of $DMT^{eff}_{gm}(F)$ such that $H^i_M(X,\q)=\ns$  for $i \ge -n$. 
Then Remark \ref{tappl}(2) 
 yields that $L(X\otimes \q)$ belongs to $K^b(\q-\operatorname{vect})_{w\ge n+1} \cong DMT^{o}_{gm,w\ge n+1}(F)\otimes \q$ (recall that  $\q-\operatorname{vect}$ is the category of finite-dimensional $\q$-vector spaces). 
Certainly, Proposition \ref{torsindep}(1) (and also the previous part of the current remark) implies that there exists a distinguished triangle $T \to X \to X' \to T[1]$ in $DMT^{eff}_{gm}(F)$ such that $T$ is a torsion motive and the groups $H^i_M(X',\z/m)$ vanish for $i\le n, m \in \z$. 

Roughly speaking, this means that one can cut off 
the "torsion part" of weight $0$ motivic cohomology from a Tate motive. 

As in Remark \ref{tappl}(2) we can consider $H^i_M(-,\q/\z)$ instead of $H^i_M(-,\z/m)$ for all $m$.
\end{rema}



%

\begin{rema}\label{rkacl}
Our results also imply that certain extension-closed subsets of triangulated categories turn out to be Karoubi-closed. 

Theorem \ref{dirsum} implies that the extension-closure of the set $((\obj \du) \cap \cu_{w\le 0}) \cup (\cu_{w\ge 0})$ is Karoubi-closed.

For a system of subcategories  $\du_i$, integers $N$ and $N'$, and a non-decreasing sequence $a_i$ (as in Proposition \ref{torsindep})
one has the extension-closure of the set $\bigcup\limits_{i\in \z} ((\obj \du_{a_i}) \cap \cu_{w= -i}) \cup \cu_{w\ge -N'}$   being Karoubi-closed. 
Indeed, the set is equal to the extension-closure of the set $\bigcap\limits_{i\in \z} ((\obj \du_{a_i}) \cap \cu_{w\le -i}) \cup (\cu_{w\ge -i+1}) \cup \cu_{w\ge -N'}$. The latter is Karoubi-closed as it is an intersection of Karoubi-closed subsets.


Moreover, if $a_l=0$ for some $l$ one certainly obtains that the extension-closure of the set $\bigcup\limits_{i\in \z} ((\obj \du_{a_i}) \cap \cu_{w= -i})$ is Karoubi-closed.


\end{rema}


\begin{thebibliography}{1}



\bibitem[BaS01]{bashli} Balmer P.,  Schlichting M., Idempotent completion
of triangulated categories// J. of Algebra 236, no. 2 (2001), 819-834.

 \bibitem[BBD82]{bbd} Beilinson A., Bernstein J.,
Deligne P., Faisceaux pervers, Asterisque 100, 1982, 5--171.


\bibitem[Bon10]{bws} Bondarko M.V.,
Weight structures vs.  $t$-structures; weight filtrations,
 spectral sequences, and complexes (for motives and in general)//
 J. of K-theory, v. 6, i.	03, pp. 387--504, 2010,
see also \url{http://arxiv.org/abs/0704.4003}


\bibitem[Bon11]{bzp} Bondarko M.V., $\mathbb{Z}[\frac{1}{p}]$-motivic resolution of singularities//  Compositio Math., vol. 147, iss. 5, 2011, 1434--1446.



\bibitem[Bon16]{bpure} Bondarko M.V., On  torsion pairs,  (well generated) weight structures,  adjacent $t$-structures, and related    (co)homological functors, 
preprint, \url{https://arxiv.org/abs/1611.00754}

\bibitem[BoI15]{bondiv}  Bondarko M.V., Ivanov M.A., On Chow weight structures for $cdh$-motives with integral coefficients//
Algebra i Analiz, v. 27 (2015), i. 6, 14--40, see also  St. Petersburg Mathematical Journal 27.6 (2016), 869--888.


\bibitem[BoK17]{bokum}  Bondarko M.V., Kumallagov D.Z., On Chow weight structures without projectivity and resolution of singularities, a manuscript  submitted to Algebra i Analiz. 


\bibitem[BoS14]{bosocwh}  Bondarko M.V., Sosnilo V.A., Detecting the $c$-effectivity of motives, their weights, and dimension via Chow-weight (co)homology: a "mixed motivic decomposition of the diagonal", preprint, \url{http://arxiv.org/abs/1411.6354}

\bibitem[BoS16a]{bososn}  Bondarko M.V., Sosnilo V.A., Non-commutative localizations of  additive categories and weight structures; applications to birational motives//  J. Math. Jussie,  published online: 24 May 2016,  DOI: \url{http://dx.doi.org/10.1017/S1474748016000207}  

\bibitem[BoS16b]{bonspkar} Bondarko M.V., Sosnilo V.A., On constructing weight structures and extending them to idempotent extensions, 
to appear in Homology, Homotopy and Appl., \url{http://arxiv.org/abs/1605.08372}

\bibitem[BoS17]{bsnew} Bondarko M.V., Sosnilo V.A., On purely generated smashing weight structures and weight-exact localizations,  a manuscript in preparation.



\bibitem[Kel12]{kellymotcharp} Kelly S., Triangulated categories of motives in positive characteristic, PhD thesis of Universit\'e Paris 13 and of the Australian National University,  
2012, \url{http://arxiv.org/abs/1305.5349}


\bibitem[Nee01]{neebook} Neeman A. Triangulated Categories. Annals of Mathematics Studies 148 (2001), Princeton University Press, viii+449 pp.


\bibitem[Pau08]{konk} Pauksztello D., Compact cochain objects in
triangulated categories and
co-t-structures//     Central European Journal of Mathematics,
vol. 6, n. 1, 2008, 25--42.







 
 


\bibitem[Voe10]{canc} Voevodsky V., Cancellation theorem,  Doc. Math. Extra Volume in honor of A. Suslin (2010), 671--685. 



\bibitem[Wil11]{wildint} Wildeshaus J., Motivic intersection complex, in: Regulators, Contemp. Math 571 (2011), 255--276.



\end{thebibliography}
\end{document}